\documentclass{amsart}

\usepackage{latexsym,amssymb,amsthm,amsmath}

\theoremstyle{plain}

\newtheorem{theorem}{Theorem}
\newtheorem{lemma}{Lemma}
\newtheorem{proposition}{Proposition}
\newtheorem{corollary}{Corollary}

\theoremstyle{remark}

\newtheorem{remark}{Remark}

\def\R{\mathbb{R}}
\def\C{\mathbb{C}}

\def\H{\mathbb{H}}
\def\S{\mathbb{S}}
\def\M{\mathbb{M}}
\def\Z{\mathbb{Z}}

\begin{document}

\title[Totally umbilical hypersurfaces] {Totally umbilical hypersurfaces\\ of
manifolds admitting a unit Killing field}

\author[R. Souam]{Rabah Souam}
\address{Institut de Math\'ematiques de Jussieu, CNRS UMR 7586\\ Universit\'e Paris
Diderot -- Paris 7, ``G\'eom\'etrie et Dynamique''\\ Site Chevaleret, Case 7012\\
75205 -- Paris Cedex 13\\ France}
\email{souam@math.jussieu.fr}

\author[J. Van der Veken]{Joeri Van der Veken}
\address{Katholieke Universiteit Leuven\\ Departement
Wiskunde\\ Celestijnenlaan 200 B -- Box 2400\\ BE-3001 Leuven\\ Belgium}
\email{joeri.vanderveken@wis.kuleuven.be}
\thanks{The second author is a post-doctoral researcher supported by the Research Foundation -- Flanders (F.W.O.)}
\thanks{This work was done while the second author visited the Universit\'e Paris Diderot -- Paris 7 supported by a grant of the Research Foundation -- Flanders (F.W.O.)}

\begin{abstract}
We prove that a Riemannian product of type $\M^n \times \R$ admits totally umbilical
hypersurfaces if and only if $\M^n$ has locally the structure of a warped product and we
give a complete description of the totally umbilical hypersurfaces in this case. Moreover,
we give a necessary and sufficient condition under which a Riemannian three-manifold carrying
a unit Killing field admits totally geodesic surfaces and we study local and global properties 
of three-manifolds satisfying this condition.
\end{abstract}

\keywords{totally umbilical, totally geodesic, product manifold, Killing field, warped product}

\subjclass[2000]{}

\maketitle

%%%%%%%%%%%%%%%%%%%%%%%%%%%%%%%%%%%%%%%%%%%%%%%%%%%%%%%%%%%%%%%%%%%%%%%%%%%%%%%%%%%%%%%%%%%%%%%%%%%%%%%%%%%%%%%%%%%%%%%

\section{Introduction}

The starting point of the research leading to this paper was the
classification of totally umbilical surfaces in three-dimensional
homogeneous spaces with a four-dimensional isometry group, which 
can be found in \cite{ST2} and \cite{VdV}. As is well-known, these
three-spaces admit Riemannian submersions onto surfaces of constant
Gaussian curvature and the unit vector field tangent to the fibers is
Killing. It turns out that such a space admits totally umbilical 
surfaces if and only if it is a Riemannian product of the base
surface and the fibers, i.e., if and only if its universal covering
is either $\S^2(\kappa) \times \R$ or $\H^2(\kappa) \times \R$. Moreover, the obtained
classification of totally umbilical surfaces was extended to a 
classification of totally umbilical hypersurfaces of the conformally
flat symmetric manifolds $\S^n(\kappa) \times \R$ and $\H^n(\kappa) \times \R$ in 
\cite{VdVV} and \cite{CKVdV}.

Two questions for further generalizations come up naturally now.
\begin{enumerate}
\item When does a Riemannian product of type $\M^n \times \R$ admit totally umbilical hypersurfaces and what are they?
\item When does a Riemannian three-space with a unit Killing field admit totally umbilical surfaces and what are they?
\end{enumerate}

In this paper, we give a complete answer to the first question. 
Our Theorem \ref{theo1} states that a necessary and sufficient 
condition for $\M^n \times \R$ to admit totally umbilical hypersurfaces
is that $\M^n$ itself has (locally) the structure of a warped product.
Moreover, we give a full description of all totally umbilical 
hypersurfaces of such a manifold and we remark that our results
are still valid if we start with a warped product instead of with 
a Riemannian product as ambient space.

For the second question we give a partial answer. We find a necessary
and sufficient condition for a three-manifold with a unit Killing field to 
admit totally geodesic surfaces. Remark that it is not necessary
that the three-manifold reduces to a Riemannian product, a fact which is 
already illustrated by the standard three-sphere. We describe the
totally geodesic surfaces and we study the local and global properties
of three-spaces satisfying our condition.

We are grateful to Eric Toubiana for valuable remarks on a first version of this work.

%%%%%%%%%%%%%%%%%%%%%%%%%%%%%%%%%%%%%%%%%%%%%%%%%%%%%%%%%%%%%%%%%%%%%%%%%%%%%%%%%%%%%%%%%%%%%%%%%%%%%%%%%%%%%%%%%%%%%%%

\section{Preliminaries}

Let $(M^n,g) \to (\tilde M^{n+1}, \tilde g)$ be an isometric immersion
between Riemannian manifolds. If $N$ is a unit normal 
vector field along the immersion and $\nabla$ and $\tilde \nabla$ are the Levi-Civita 
connections of $(M^n,g)$ and $(\tilde M^{n+1}, \tilde g)$, then the 
second fundamental form $h$ and the shape 
operator $S$ associated to $N$ are defined by the formulas of Gauss and Weingarten: for any 
vector fields $X,Y$ on $M^n$ one has
$$ \tilde\nabla_X Y = \nabla_X Y + h(X,Y)N, \quad SX = -\tilde\nabla_X N. $$
It is easy to check that $S$ is a symmetric $(1,1)$-tensor field on $M^n$,
which is related to $h$ by $h(X,Y)=g(SX,Y)$. 
We call the immersion \emph{totally umbilical} if $S$ is a multiple of 
the identity at every point and we call it \emph{totally geodesic} if 
$S$ vanishes identically.

In our results, some special types of vector fields on Riemannian manifolds
will occur. Let $(M,g)$ be a Riemannian manifold and let $\xi$ be a vector
field on $M$. Then $\xi$ is said to be \emph{Killing} if and only if 
$\mathcal{L}_{\xi}g = 0$, where $\mathcal L$ is the Lie derivative. This
condition means that the flow of $\xi$ consists of isometries, and in terms
of the Levi-Civita connection $\nabla$ one can reformulate it as
$$ g(\nabla_X\xi,Y) + g(\nabla_Y\xi,X) = 0 $$
for all $p \in M$ and $X,Y \in T_pM$. 

More generally, $\xi$ is said to be \emph{conformal} if and only if 
$\mathcal{L}_{\xi}g = 2 \phi g$ for some function $\phi$. This
means that the flow of $\xi$ consists of conformal maps.

Finally, we say that $\xi$ is \emph{closed conformal} if and only if
it is conformal and its dual one-form is closed. It can be checked by 
a straightforward computation that $\xi$ is closed conformal if and only if
$$ \nabla_X \xi = \phi X, $$
for all $p \in M$ and all $X \in T_pM$, where $\phi$ is as above.

In all what follows the manifolds will be assumed of class 
$C^\infty$.

%%%%%%%%%%%%%%%%%%%%%%%%%%%%%%%%%%%%%%%%%%%%%%%%%%%%%%%%%%%%%%%%%%%%%%%%%%%%%%%%%%%%%%%%%%%%%%%%%%%%%%%%%%%%%%%%%%%%%%%

\section{Totally umbilical hypersurfaces of manifolds of type $\M^n\times I$} \label{sec - product}

Denote by $\M^n\times I$ the Riemannian product of a Riemannian
manifold $(\M^n, g_{\M^n})$ and an open interval $I$ of the
Euclidean line and let $\pi: \M^n\times I  \to \M^n$ be the
canonical projection. We shall denote by $\xi$ a unit vector field
on $\M^n \times I$, tangent to the fibres of $\pi$. Remark that
$\xi$ is a unit Killing field.

There are two natural families of examples of totally geodesic hypersurfaces
of $\M^n \times I$, namely the slices $\M^n \times \{t_0\}$, $t_0 \in I$ and
the inverse images under $\pi$ of totally geodesic hypersurfaces of $\M^n$, 
if they exist. We are thus interested in totally umbilical hypersurfaces
which are at some point neither orthogonal nor tangent to $\xi$.

If $\Sigma$ is a hypersurface of $\M^n\times I $ with unit normal
$N$, one can define a vector field $T$ and a real-valued function
$\nu$ on $\Sigma$ by the following orthogonal decomposition of
$\xi$:
\begin{equation} \label{eq - decomposition1}
\xi = T + \nu N.
\end{equation}
Then $T$ and $\nu$ satisfy the following equations.

\begin{lemma} \label{lem1}
Let $\Sigma$ be a hypersurface in $\M^n\times I$ and denote by
$\nabla$ the Levi-Civita connection of $\Sigma$, by $S$ the shape
operator of the immersion and by $h$ the second fundamental form.
Then for any vector $X$ tangent to $\Sigma$:
 $$ \nabla_X T = \nu SX, \quad X(\nu) = -h(X,T). $$
\end{lemma}
\begin{proof}
Denote by $\widetilde{\nabla}$ the Levi-Civita connection of $\M^n
\times I$. Since $\xi$ is a parallel vector field on $\M^n \times
I$, one has $\widetilde{\nabla}_X \xi=0$. Now let $X$ be tangent to
$\Sigma$. By the definitions of $T$ and $\nu$ and by using the
formulas of Gauss and Weingarten, it follows that
$$0 = \tilde\nabla_X \xi = \tilde\nabla_X (T + \nu N) = \nabla_X T + h(X,T)N + X(\nu)N - \nu SX.$$
The result follows by considering the tangent, resp. normal,
component of the above equation.
\end{proof}

If $\Sigma$ is totally umbilical in $\M^n \times I$, say $S =
\lambda \, \mathrm{id}$, then $T$ is a closed conformal field on
$\Sigma$. Indeed, in this case it follows from Lemma \ref{lem1}
above that $\nabla_X T = \nu \lambda X$ for all vector fields $X$
tangent to $\Sigma$. We shall now prove that if $\Sigma$ is
non-vertical and non-horizontal at some point $p$, we can use $T$ to
construct a local non-vanishing conformal field on $\M^n$.

\begin{proposition} \label{prop1}
Let $\Sigma$ be a totally umbilical hypersurface of  $\M^n\times I$,
which is neither vertical  nor horizontal at some point. Then the
canonical projection $\pi:\M^n\times I\to\M^n$ is locally a
diffeomorphism between an open neighborhood $U$ of this point in
$\Sigma$ and the open subset $\pi U$ of $\M^n$. Let $T$ be as above
and denote by $T_0$ be the projection of $T$ to $\pi U$, rescaled
such that it has the same length as $T$ again. Then $T_0$ is a
closed conformal field on $\pi U$.
\end{proposition}

\begin{proof}
Let $\Sigma$ be a totally umbilical hypersurface of $\M^n\times I$.
Let $\xi$, $N$, $T$  and $\nu$ be as above and assume that the shape
operator associated to $N$ is $S = \lambda\,\mathrm{id}$. Suppose
$\Sigma$ is non-vertical and non-horizontal at some point, then it
is clear that there is an open neighborhood $U$ of this point in
$\Sigma$ where $\nu$ does not vanish and such that $\pi$ is a local
diffeomorphism between  $U$ and its image $\pi U$ in $\M^n$. First,
extend the vector fields $T$, $N$ and the functions $\nu$, $\lambda$
to the whole of $\pi U \times I$ by using the one-parameter group of
translations corresponding to the Killing field $\xi$ and denote
these again by $T$, $N$, $\nu$ and $\lambda$. Since $\nu$ and
$\lambda$ are constant on fibres of $\pi$, one can also view them as
functions on $\pi U$. Using these notations, the vector field $T_0$
on $\pi U$ is
$$ T_0 = (d\pi)(T) \frac{\|T\|}{\|(d\pi)(T)\|} = \frac{1}{\nu} (d\pi)(T) $$
and its horizontal lift to $\pi U \times I$ is
$$ \widetilde{T_0} = \nu \xi - N= \frac{\nu^2-1}{\nu}\xi + \frac{1}{\nu}T. $$
Remark that $T_0$ is, up to the sign, the projection of $N$ to $\pi
U$.

Now let $X$ be a vector field on $\pi U$ and denote by $\widetilde
X$ its horizontal lift. Then
\begin{equation} \label{eq - 1}
\nabla^{\M^n}_X T_0 = (d\pi)(\widetilde{\nabla}_{\widetilde X}
\widetilde{T_0}) = (d\pi)(\widetilde{\nabla}_{\widetilde X}(\nu \xi
- N)) = -(d\pi)(\widetilde{\nabla}_{\widetilde{X}}N)
\end{equation}
Let $Y$ be a local vector field on $\Sigma$ such that $(d\pi)(Y)=X$.
Denote the extension to $\pi U \times I$, using the flow of $\xi$,
again by $Y$. Then $\widetilde X = Y - \langle Y, \xi \rangle \xi$
and
\begin{equation} \label{eq - 2}
\widetilde{\nabla}_{\widetilde X}N = \widetilde{\nabla}_{Y - \langle
Y, \xi \rangle \xi}N = \widetilde{\nabla}_Y N - \langle Y, \xi
\rangle \widetilde{\nabla}_{\xi}N = - \lambda Y.
\end{equation}
Here we used that $[\xi,N]=0$ implies $\widetilde{\nabla}_{\xi} N =
\widetilde{\nabla}_N \xi = 0$. From \eqref{eq - 1} and \eqref{eq -
2}, we obtain $\nabla_X^{\M^n}T_0 = \lambda X$, which proves that
$T_0$ is indeed closed conformal.
\end{proof}

The fact that $\M^n$ admits a local closed conformal field,
determines locally its Riemannian structure, as shown by the
following known result (see \cite{M, RU} and the references therein).

\begin{proposition} \label{prop2}
Let $V$ be a local closed conformal field without zeros on a
Riemannian manifold $\M^n$, say $\nabla^{\M^n}_X V = f X$ for some
non-vanishing function  $f$ and for all vector fields $X$ on $\M^n$.
Then $\M^n$ has locally the structure of a warped product of an
interval of the Euclidean line with some $(n-1)$-dimensional
Riemannian manifold.
\end{proposition}

\begin{proof}
One can check that the distribution orthogonal to $V$ is integrable
and hence one can find a local coordinate system $(x_1, \ldots,
x_n)$ on $\M^n$ such that $\partial_{x_1} = V$ and $\partial_{x_j}$
is orthogonal to $\partial_{x_1}$ for $j \geq 2$. With respect to
these coordinates, the metric on $\M^n$ takes the form
$$ g = g_{11}(x_1, \ldots, x_n) dx_1^2 + \sum_{i,j=2}^n g_{ij}(x_1, \ldots, x_n) dx_i dx_j. $$
It follows from a straightforward computation that
$\partial_{x_j}g_{11}=0$ for $j \geq 2$ and that $\partial_{x_1}
g_{ij} = 2 f g_{ij}$ for $i,j \geq 2$. Hence, one has
$$ g = g_{11}(x_1) dx_1^2 + \exp\left( 2 \int f \, dx_1 \right) \sum_{i,j=2}^n c_{ij}(x_2, \ldots, x_n) dx_i dx_j. $$
To conclude, we prove that $\partial_{x_j}f = 0$ for $j \geq 2$,
such that, after a change of the $x_1$-coordinate, the metric above
is indeed a warped product metric. To see this, let  $R$
be  the curvature tensor of $\M^n$
, then
\begin{align*}
0 &= \langle R(\partial_{x_1}, \partial_{x_j}) \partial_{x_1}, \partial_{x_1} \rangle = \langle \nabla^{\M^n}_{\partial_{x_1}} \nabla^{\M^n}_{\partial_{x_j}} \partial_{x_1} - \nabla^{\M^n}_{\partial_{x_j}} \nabla^{\M^n}_{\partial_{x_1}} \partial_{x_1} , \partial_{x_1} \rangle \\
&= \langle (\partial_{x_1}f)\partial_{x_j} - (\partial_{x_j}f)\partial_{x_1} , \partial_{x_1} \rangle = -(\partial_{x_j}f)g_{11}.
\end{align*}\end{proof}

\begin{remark}
The converse to Proposition \ref{prop2} is also true.
In a warped product $I\times_f M = (I \times M, dt^2 +f(t)^2 g_M)$, the field $f(t) {\partial_ t}$ is
closed conformal and  vanishes nowhere.
\end{remark}

We can now prove our main result in this section.

\begin{theorem}\label{theo1}
A Riemannian product space $\M^n\times I$ admits a totally umbilical
hypersurface $\Sigma$, which is neither vertical nor horizontal at
some point $(p,t)\in \M^n\times I$, if and only if $\M^n$ has in a neighborhood of  $p$ the structure of a
warped product of an interval of the Euclidean line with some
$(n-1)$-dimensional Riemannian manifold.

In particular, when $n=2,$ there exists a totally umbilical surface in $\M^2\times I,$ which is neither vertical nor horizontal at some point $(p,t)\in \M^2\times I,$ if and only if $\M^2$ admits a non zero Killing field in a neighborhood of $p.$ Moreover any 
such surface is invariant by a local one-parameter group of local isometries of  $\M^2\times I$ keeping the factor $I$ pointwise fixed. 
\end{theorem}

\begin{proof}
It follows from Propositions \ref{prop1} and \ref{prop2} above that
$\M^n$ having  the structure of a warped product in a neighborhood of $p$ is a
necessary condition for $\M^n\times I$ to admit a totally umbilical
hypersurface which is non-vertical and non-horizontal at $(p,t).$

We shall now prove that this condition is also sufficient. Assume
that $\M^n = J \times_f M^{n-1}$, i.e., that the metric on
$\M^n$ can be written as
$$ g_{\M^n} = dx_1^2 + f(x_1)^2 g_{M^{n-1}}(x_2,\ldots,x_n). $$
Then the metric on $\M^n\times I$ can be written as
$$ g = dx_0^2 + dx_1^2 + f(x_1)^2 g_{M^{n-1}}(x_2,\ldots,x_n). $$

We know from above that a non-vertical and non-horizontal totally
umbilical hypersurface $\Sigma$ of $\M^n \times I$ should be tangent
to the distribution orthogonal to the vector fields $\partial_{x_0}$
and $\partial_{x_1}$ at any of its points. This means that $\Sigma$
is generated by a curve in the $(x_0,x_1)$-plane, say
$\alpha(s)=(x_0(s),x_1(s))$. Assume that $\alpha$ is parametrized by
arc length, then there exists a function $\theta$ such that
$$ x_0'(s) = \sin\theta(s), \quad x_1'(s) = \cos\theta(s). $$
In this case, the tangent space to $\Sigma$ is spanned by $X_1 =
\sin\theta(s) \partial_{x_0} + \cos\theta(s) \partial_{x_1}$, $X_2 =
\partial_{x_2}, \ldots, X_n = \partial_{x_n}$ and a unit normal to
$\Sigma$ is given by $N = \cos\theta(s) \partial_{x_0} -
\sin\theta(s) \partial_{x_1}$.

One can compute the Levi-Civita connection $\widetilde{\nabla}$ of
$\M^n\times I$ from the metric above to verify
$$ \widetilde{\nabla}_{X_1} N = -\theta'(s) X_1, \qquad \widetilde{\nabla}_{X_j} N = - \sin\theta(s)\frac{f'}{f}X_j $$
for every $j\geq 2$. Hence, $\Sigma$ is totally umbilical if and
only if
$$ \theta'(s) = \sin\theta(s) \frac{f'}{f}. $$

One can now use this equation to determine the functions $x_0(s)$
and $x_1(s)$. Indeed, we have
$$ x_1''(s) = -\sin\theta(s) \theta'(s) = -\sin^2\theta(s) \frac{f'(x_1(s))}{f(x_1(s))} = -(1-x_1'(s)^2) \frac{f'(x_1(s))}{f(x_1(s))}, $$
which yields after a first integration
$$ x_1'(s) = \pm \sqrt{1-c^2 f(x_1(s))^2} $$
for some real constant $c$. This ODE for $x_1(s)$ is, at least
locally, always solvable. The function $x_0(s)$ is then determined
by
$$ x_0(s) = \int \sqrt{1-x_1'(s)^2} \, ds = \int c f(x_1(s)) \, ds. $$
Thus there does always exist a non-vertical and non-horizontal
totally umbilical hypersurface of $\M^n\times I$ if $\M^n$ is
locally isometric to the warped product described above. A
parametrization for such a totally umbilical hypersurface $\Sigma$
is $\varphi(s,u_1,\ldots,u_{n-1}) = (x_0(s) , x_1(s), u_1, \ldots,
u_{n-1}).$

In the particular case when $n=2$, observe the following general fact that can be checked straightforwardly. Let $J$ denote the rotation over
90 degrees of an oriented Riemannian surface $M^2$, which is locally well-defined on any Riemannian surface $M^2$. Then a vector field $X$
on $M^2$ is closed conformal if and only if $JX$ is Killing. Hence, if $\Sigma$ is a totally umbilical surface in $\M^2 \times I$, then $JT$ is a Killing field on $\Sigma$. Moreover, $JT$ is orthogonal to the fibers of $\pi$ and $(d\pi)(JT)=JT_0$ is a Killing field on $\M^2$. This implies the result. 

\end{proof}

In  particular, from Theorem \ref{theo1}, we recover the
classification of totally umbilic surfaces in $\S^2\times \R$ and
$\H^2\times \R$ obtained in \cite{ST2} and \cite{VdV}.

\begin{remark}
Theorem \ref{theo1} also gives a classification of totally umbilical
hypersurfaces in Riemannian product spaces of type $\M^n \times I$.
\end{remark}

\begin{remark}
As a further particular case, it is interesting to observe that the
results above provide a new proof for the classification of totally
umbilic hypersurfaces in the Euclidean space $\R^{n+1}$ since the
latter can be viewed (in various ways) as a product $\R^n\times \R.$
The present proof has the advantage to work assuming only a $C^2$
regularity for the hypersurfaces. It can indeed easily be checked
that $C^2$ regularity for the hypersurfaces is enough in the above
results. The standard proof of the classification of totally umbilic
hypersurfaces in $\R^{n+1}$ works for hypersurfaces with at least
$C^3$ regularity. It is known this classification also holds for
$C^2$-hypersurfaces (see \cite{Pa} and \cite{ST1}).
\end{remark}

\begin{corollary} \label{cor1}
A Riemannian warped product space $I \times_f \M^n$ admits a totally
umbilical hypersurface $\Sigma$, which is neither vertical nor
horizontal at some point $(t,p)\in I \times_f \M^n,$ if and only if  $\M^n$ has in a neighborhood of $p$ the
structure of a warped product of an interval of the Euclidean line
with some $(n-1)$-dimensional Riemannian manifold.
\end{corollary}

\begin{proof}
First observe that a Riemannian warped product space $I \times_f
\M^n$ is conformal to a Riemannian product $J\times \M^n$, for some
open interval $J$ of $\R.$  Indeed, with some abuse of notation, the
metric on $I \times_f \M^n$ writes
$$g = dt^2 + f(t)^2 g_{\M^n} = f(t)^2 \left(\frac{dt^2}{f(t)^2} +  g_{\M^n} \right). $$
Choosing a new parameter $s=\psi(t)=\int \! \frac{dt}{f(t)}$ and
introducing the function $h(s)= f(\psi^{-1}(s)),$ we can write:
$$g = h(s)^2 (ds^2 + g_{\M^n})$$
as desired.

Now the claim follows from Theorem \ref{theo1} and the known fact
that totally umbilic hypersurfaces are preserved under conformal
diffeomorphisms between  the ambient manifolds.
\end{proof}

%%%%%%%%%%%%%%%%%%%%%%%%%%%%%%%%%%%%%%%%%%%%%%%%%%%%%%%%%%%%%%%%%%%%%%%%%%%%%%%%%%%%%%%%%%%%%%%%%%%%%%%%%%%%%%%%%%%%%%%

\section{Totally geodesic surfaces in a three-dimensional space \\ admitting a unit Killing field}

In this section we will characterize locally the Riemannian three-manifolds admitting a unit Killing field
which possess totally geodesic surfaces. A good reference on manifolds admitting a Killing field of constant 
length is given by \cite{BN}.

We start with a result which is valid in all dimensions. Let $M$
denote a Riemannian manifold which admits a unit Killing field
$\xi$.  The product manifolds $\M^n\times I$ considered in Section
\ref{sec - product} are  a particular case. Denote by
$\widetilde{\nabla}$ the Levi-Civita connection of $M.$ Let $\Sigma$
be a hypersurface in $M$ with unit normal $N$. Then we can, as in
the previous section, define a vector field $T$ and a real-valued
function $\nu$ on $\Sigma$ by the orthogonal decomposition
\begin{equation} \label{eq - decomposition2}
\xi = T + \nu N.
\end{equation}

The following result is a key fact for our purposes:

\begin{proposition} \label{prop3}
Let $\Sigma$ be a totally geodesic hypersurface in a Riemannian
manifold $M$ admitting a unit Killing field $\xi.$ Suppose $\xi$ is
not tangent to $\Sigma$ at some point. Then one can extend the
vector field $T$ to a neighborhood of this point in $M$ using the
local flow of $\xi$. If one denotes the resulting vector field again
by $T$, then $T$ is a local Killing field on $M$.
\end{proposition}

\begin{proof}
Since $\xi$ is transversal to $\Sigma$ in a neighborhood of the
point, using the local flow $(\phi_t)_{t\in I}$ of $\xi$ we obtain a
foliation $\mathcal F$ of an open subset of $M$ by the totally
geodesic hypersurfaces $\phi_t(\Sigma).$ In this way we have local
extensions of the  fields $T$ and $N$ and the function $\nu$. Denote
these extensions again by $T$, $N$ and $\nu$. Note that (\ref{eq -
decomposition2}) is again valid for these extensions.

We have to verify that $\langle \widetilde{\nabla}_X T , Y \rangle +
\langle X , \widetilde{\nabla}_Y T \rangle = 0 $ for all vector
fields $X$ and $Y$.

First we note that
\begin{equation} \label{eq - derivative}
T(\nu)= 0.
\end{equation}
Indeed, since the hypersurfaces $\phi_t(\Sigma)$ are totally
geodesic and $\xi$ is a unit Killing field, we have
$$ T(\nu)= T(\langle \xi, N\rangle) = \langle \widetilde{\nabla}_T \xi, N\rangle = \langle \widetilde{\nabla}_T \xi, \frac{1}{\nu} (T-\xi)\rangle = 0. $$

Consider now a vector field $X$ which is tangent to the leaves of
the foliation $\mathcal F.$ Then
\begin{equation}\label{eq - tangent}
\widetilde{\nabla}_{X} T = \widetilde{\nabla}_{X} \xi - X(\nu) N.
\end{equation}
Furthermore,
$$ \widetilde{\nabla}_{N} T =\widetilde{\nabla}_{\frac{1}{\nu} (\xi-T)} T=  \frac{1}{\nu} (\widetilde{\nabla}_{T}\xi- \widetilde{\nabla}_{T} T ). $$
Using (\ref{eq - tangent}) and (\ref{eq - derivative}) we get
\begin{equation}\label{eq - normal}
\widetilde{\nabla}_{N} T = 0.
\end{equation}

Using (\ref{eq - tangent}) and (\ref{eq - normal}), it is easy to
check that $T$ is Killing.
\end{proof}

We now particularize to the case where the ambient manifold is three-dimensional.
Well-known examples of such manifolds are Riemannian products of type $\M^2 \times \R$ and also the unit three-sphere $\S^3$, Berger spheres and the Thurston spaces $\widetilde{\mathrm{SL}}(2,\R)$ and $\mathrm{Nil}_3$. In the next section we will describe more  such spaces.
We first prove an important basic formula.

\begin{lemma} \label{lem2}
Let $M^3$ be an oriented Riemannian manifold carrying a unit Killing
field $\xi$. Denote by $\widetilde{\nabla}$ the Levi-Civita
connection of $M^3$ and by $\times$ its cross product. Then there
exists a real-valued function $\tau$ on $M^3$, with $\xi(\tau) = 0$,
such that
$$ \widetilde{\nabla}_X \xi = \tau (X \times \xi) $$
for all vector fields $X$ on $M^3$.
\end{lemma}

\begin{proof}
It is clear that $\widetilde{\nabla}_X \xi$ is perpendicular to $\xi$ and $X$ since $\xi$ is a unit Killing field.
Because the space is three-dimensional, we obtain that $\widetilde{\nabla}_X \xi = \tau(X) (X \times \xi)$ for some real number $\tau(X)$.

Since the mapping $T_p M^3 \to T_p M^3 : X \mapsto
\widetilde{\nabla}_X \xi = \tau(X) (X \times \xi)$ must be linear
for every point $p$, it is easily seen that $\tau$ can only depend
on the choice of $p \in M^3$ and not on the choice of $X \in
T_pM^3$. Hence $\tau$ is a real-valued function on
$M^3$.

To see that this function satisfies $\xi(\tau) = 0$, let
$(\phi_t)_{t\in I}$ be the local flow of $\xi$ as above.
Then $\widetilde{\nabla}_{(d\phi_t)X} (d\phi_t)\xi =
(d\phi_t)(\widetilde{\nabla}_X \xi)$, or, equivalently,
$\tau(\phi_t(p)) ((d\phi_t)X \times \xi) = \tau(p)
((d\phi_t)X \times \xi)$ for every parameter $t$ and for every
$p\in M^3$ and $X \in T_p M^3$. We conclude that
$\tau(\phi_t(p))= \tau(p)$ and hence that $\xi(\tau) = 0$.
\end{proof}

The first main result in this section is the following.

\begin{theorem} \label{theo2}
Let $M$ be a Riemannian three-manifold carrying a unit Killing field
$\xi$ and $p$ a point in $M.$ Then
\begin{itemize}
\item [(1)]
$M$ admits a totally geodesic surface  passing through $p$ which is  everywhere orthogonal to $\xi$  if and only if  $M$ has in a neighborhood of $p$ a Riemannian product structure $\Sigma\times I$ of some surface $\Sigma$ with an interval $I$ and  $\xi$ is tangent to the factor $I.$
\item [(2)]
$M$ admits a totally geodesic surface passing through $p$ which is everywhere tangent to $\xi$ if and only if there exists a geodesic through $p$ 
in $M$ on which $\tau$ vanishes and which is orthogonal to $\xi$ at $p.$
\item [(3)] The following three assertions are equivalent.
\begin{itemize}
\item[(i)] $M$ admits a totally geodesic surface passing through $p$ which
is neither orthogonal nor tangent to $\xi$ at $p.$
\item[(ii)] There is in a neighborhood of $p$ in  $M$ an orthogonal decomposition $\xi = X_1 +
X_2$ of $\xi$, where $X_1$ and $X_2$ are  Killing fields without zeros
that commute.
\item[(iii)] There exist local coordinates $(x,y,z)$ around $p$ in $M$ 
with $\xi = \partial_y + \partial_z$, such that the metric
takes the form
$$ g = dx^2 + \sin^2\theta(x) dy^2 + \cos^2\theta(x) dz^2. $$
\end{itemize}
\end{itemize}

\end{theorem}

\begin{remark}
Statement (1) is valid in all dimensions. In dimension 3 it is
equivalent to the vanishing of $\tau$ in a neighborhood of $p.$
\end{remark}

\begin{proof}
We denote by $(\phi_t)_{t\in I}$ the local flow of $\xi.$

\medskip

(1) Suppose $M$ admits a totally geodesic surface $\Sigma,$ passing through $p,$ which is everywhere orthogonal to $\xi.$ Restricting $\Sigma$ and the interval $I$ if necessary, the mapping $(x,t)\in \Sigma\times I \mapsto \phi_t(x)\in M$  parametrizes an open subset of $M$ and is easily checked to be an   isometry between the Riemannian product  manifold $\Sigma\times I $ and that  open set. 

Conversely it is clear that if $M$ has locally a product structure of some surface $\Sigma$ and an interval $I$ to which $\xi$ is tangent then the surfaces  $\Sigma\times \{t\}$ are totally geodesic and orthogonal to $\xi$ at each point.  

\medskip

(2) Suppose $s\in J\mapsto \gamma(s)$ is a geodesic of $M$ parametrized by arc length such that $\gamma(t_0)=p$ and $ \langle \gamma^\prime (t_0), \xi \rangle =0$ for some $t_0\in J$ and $\tau(\gamma(s))=0$ for all $s\in J.$ Note first that $\gamma$ is everywhere orthogonal to $\xi.$ Indeed, for all $s\in J$
$$\frac{d}{ds}\langle  \gamma^\prime(s),\xi\rangle = \langle \frac{D\gamma^\prime}{ds} (s), \xi\rangle + \langle \gamma^\prime(s), \tau(\gamma(s)) (\gamma^\prime(s)\times \xi) \rangle = 0.$$
The mapping $(s,t)\mapsto \phi_t(\gamma(s)) $ parametrizes a surface $\Sigma$  in $M$ and   $\tau|_{\Sigma} = 0$.

We now check that $\Sigma$ is totally geodesic. A unit normal field
to $\Sigma$ is the field $N(s,t)= (d\phi_t)(\gamma^\prime(s)) \times
\xi.$ Note that $\xi$ commutes with $(d\phi_t)(\gamma^\prime(s),$
so that
$$\widetilde{\nabla}_{\xi} (d\phi_t)(\gamma^\prime(s))=\widetilde{\nabla}_{(d\phi_t)(\gamma^\prime(s))} \xi= 0,$$
where we used that $\tau|_{\Sigma} = 0.$
Therefore
$$ \widetilde{\nabla}_{\xi} N = \widetilde{\nabla}_{\xi} (d\phi_t)(\gamma^\prime(s)) \times \xi + (d\phi_t)(\gamma^\prime(s)) \times  \widetilde{\nabla}_{\xi} \xi = 0.$$
Since $s \mapsto \phi_t(\gamma(s))$ is a geodesic in $M$ for each
$t,$ we have
$$ \widetilde{\nabla}_{(d\phi_t)(\gamma^\prime(s))} N = \widetilde{\nabla}_{(d\phi_t)(\gamma^\prime(s))} (d\phi_t)(\gamma^\prime(s)) \times \xi + (d\phi_t)(\gamma^\prime(s)) \times \widetilde{\nabla}_{(d\phi_t)(\gamma^\prime(s))} \xi = 0. $$
It follows that $\Sigma$ is totally geodesic.

Conversely, suppose  $M$ admits a totally geodesic surface $\Sigma$ passing through $p$ which is
everywhere tangent to $\xi.$ As geodesics on $\Sigma$ are also geodesics on $M,$ it is enough to check that $\tau|_{\Sigma} = 0.$ This is indeed the case: take an arbitrary point $q$ of $\Sigma$ and a vector $X$
tangent to $\Sigma$ and linearly independent of $\xi$. Since
$\Sigma$ is totally geodesic, the vector $\widetilde{\nabla}_X \xi =
\tau(q) (X \times \xi)$ has to be tangent to $\Sigma$, which is only
possible if $\tau(q)=0$. 

\medskip

(3) First, we prove that (i) implies (iii). Let $\Sigma$ be totally
geodesic in $M$ such that $\xi$ is not tangent to $\Sigma$ at $p$. Extend $T$, $N$, $JT= N\times T$ and $\nu$ to a neighborhood
of this $p$  in $M$ using the local flow of $\xi$. Using
equations (\ref{eq - tangent}), (\ref{eq - derivative}) and (\ref{eq
- decomposition2}), we have:
\begin{align*}
& \widetilde{\nabla}_{JT} T = \widetilde{\nabla}_{JT}\xi - (JT)(\nu)N
 = \tau [(JT \times \xi) - \langle JT\times \xi , N\rangle N]
 = \tau (JT \times \nu N)
 = \tau\nu T, \\
& \widetilde{\nabla}_T JT = \widetilde{\nabla}_T (N \times T)
 = N \times \widetilde{\nabla}_T (\xi - \nu N)
 = N \times \tau(T \times \xi)
 = \tau\nu T.
\end{align*}

It follows that
\begin{align*}
& [T,JT]=\widetilde{\nabla}_T JT - \widetilde{\nabla}_{JT} T = 0,\\
& [T,\nu N] = [T,\xi] = 0,\\
& [JT,\nu N] = [JT,\xi] = 0.
\end{align*}
Hence, we can take local coordinates $(x,y,z)$ on $M$ such that
$\partial_x = JT$, $\partial_y = T$ and $\partial_z = \nu N$. With
respect to these coordinates, the metric takes the form
$$ g = (1-\nu^2)(dx^2 + dy^2) + \nu^2 dz^2. $$
From \eqref{eq - derivative} one has $\partial_y \nu = T(\nu) = 0$
and $\partial_z \nu = \nu N(\nu) = (\xi - T)(\nu) = 0$. After a
change of the $x$-coordinate, we obtain the form for $g$ given in
the theorem.

To see that (iii) implies (ii), it suffices to take $X_1 =
\partial_y$ and $X_2 = \partial_z$.

It remains to prove that (ii) implies (i). Let $u$ be a unit vector
field perpendicular to $X_1$ and $X_2$. We shall first prove that
$u$ commutes with $X_1$ and $X_2$. By using that $u$ is
perpendicular to $X_1$ and that $X_1$ is Killing, we obtain
\begin{align*}
\langle [X_1,u], X_1 \rangle &= \langle \widetilde\nabla_{X_1} u, X_1 \rangle - \langle \widetilde\nabla_u X_1, X_1 \rangle\\
& = -\langle u, \widetilde\nabla_{X_1} X_1 \rangle - \langle \widetilde\nabla_u X_1, X_1 \rangle \\
& = 0.
\end{align*}
Furthermore, by using that $u$ is perpendicular to $X_2$ and that
$X_1$ is Killing, we find
\begin{align*}
\langle [X_1,u], X_2 \rangle &= \langle \widetilde\nabla_{X_1} u, X_2 \rangle - \langle \widetilde\nabla_u X_1, X_2 \rangle\\
& = -\langle u, \widetilde\nabla_{X_1} X_2 \rangle + \langle \widetilde\nabla_{X_2} X_1, u \rangle \\
& = -\langle [X_1,X_2],u \rangle = 0.
\end{align*}
Finally, since $\|u\|=1$ and $X_1$ is Killing,
$$ \langle [X_1,u], u \rangle = \langle \widetilde\nabla_{X_1} u, u \rangle - \langle \widetilde\nabla_u X_1, u \rangle = 0. $$
We conclude that $[X_1,u]=0$. Analogously, we can
prove that $[X_2,u]=0$.

Now consider an integral surface of the distribution spanned by $u$
and $X_1$. Of course, this surface is nowhere tangent or orthogonal
to $\xi = X_1 + X_2$. We shall prove that it is totally geodesic. It
is sufficient to verify that $\widetilde\nabla_u u$,
$\widetilde\nabla_u X_1$ and $\widetilde\nabla_{X_1} X_1$ are all
perpendicular to $X_2$. Using Koszul's formula and the facts that
$[X_1,X_2]=[X_1,u]=[X_2,u]=0$, $\langle X_1,X_2 \rangle = \langle
X_1,u \rangle = \langle X_2,u \rangle =0$ and $\langle u,u \rangle
=1$, gives immediately that $\langle \widetilde\nabla_u u, X_2
\rangle = 0$ and $\langle \widetilde\nabla_u X_1, X_2 \rangle = 0$.
Finally, using the facts that $X_1$ and $X_2$ are orthogonal and
that $X_2$ is Killing, gives $\langle \widetilde\nabla_{X_1}X_1, X_2
\rangle = -\langle X_1, \widetilde\nabla_{X_1}X_2 \rangle = 0$.
\end{proof}

\begin{remark}\label{coordinate}
For later use, we note that in the coordinates where the metric takes the given form in case (iii) we have from the proof of the theorem: $\partial_x= JT/\|T\|$, $\partial_y= T$,  $\cos\theta(x)=\langle \xi,N\rangle$ and $\|T\|= \sin\theta(x).$ 
\end{remark}

We now study further the case (3) in Theorem \ref {theo2}. We are able to determine all the totally geodesic surfaces of $M$ in a neighborhood of $p$ in this case. We will need the  following result which can be verified by straightforward
computations.

\begin{proposition}\label{prop4}
The Levi-Civita connection of the metric $g$ defined in local coordinates $(x,y,z)$ by
\begin{equation} \label{eq - metric0}
g = dx^2 + \sin^2\theta(x) dy^2 + \cos^2\theta(x) dz^2,
\end{equation}
 is given by
\begin{align*}
& \widetilde\nabla_{\partial_x}\partial_x = 0, &&
\widetilde\nabla_{\partial_x}\partial_y = \cot\theta
\theta'\partial_y,
&& \widetilde\nabla_{\partial_x}\partial_z = -\tan\theta \theta' \partial_z, \\
& \widetilde\nabla_{\partial_y}\partial_y = -\cos\theta\sin\theta
\theta' \partial_x, && \widetilde\nabla_{\partial_y}\partial_z = 0,
&& \widetilde\nabla_{\partial_z}\partial_z = \cos\theta\sin\theta
\theta' \partial_x,
\end{align*}
Setting $\xi=\partial_y+\partial_z,$ it follows that $\widetilde \nabla_X \xi = -\theta'(x) (X
\times \xi)$ for any tangent vector $X,$ where $\times$ stands for the cross product associated with the orientation given by the chart $(x,y,z).$ Moreover, the scalar
curvature of the manifold is $(\theta')^2 - 4 \cot(2\theta)
\theta''.$
\end{proposition}

Our  second main result in this section decribes the totally geodesic surfaces in $M$ in case (3)  in Theorem \ref {theo2}. It  characterizes in particular  the flat and spherical metrics. It can be compared to a result of E. Cartan (see \cite{B} p. 233). In  the  three-dimensional case, Cartan's theorem asserts  a three-dimensional Riemannian manifold with a totally geodesic surface passing through any point with any specified plane as tangent plane must be a space form. When the manifold admits a unit Killing field, our result says that the existence of very few totally geodesic surfaces suffices to characterize the space forms of non-negative curvature.
 
 \newpage
 
\begin{theorem}\label{sphere}
Let $M$ be a Riemannian three-manifold carrying a unit Killing  field $\xi$ and let $p\in M$. Suppose there is a totally geodesic surface $\Sigma_1$ passing through $p$ which is neither orthogonal nor tangent to $\xi$ at $p.$ Then:

\begin{itemize}
\item [(1)] There is a second totally geodesic surface  passing through $p$ which is orthogonal to $\Sigma_1.$ 
\item[(2)] If there exists a third  totally geodesic surface through $p$ which is not tangent to $\xi$ at $p,$ 
then $M$ has constant non-negative sectional curvature in  a neighborhood of $p$ and thus is around $p$ isometric to 
an open subset of the sphere $\Bbb S^3$ with a metric of constant curvature or  to an open set of the Euclidean space $\Bbb R^3.$
\item[(3)] If $M$ does not have constant positive curvature near $p,$ then there exists a totally geodesic surface through $p$ which is tangent to $\xi$ at $p$ if and only if $\tau(p)=0.$
\end{itemize}
\end{theorem}

\begin{proof} From  case (3)  in Theorem \ref {theo2} we can find local coordinates $(x,y,z)$ in a neighborhood $W$ of $p$ where the metric takes the form \eqref{eq - metric0}, the point $p$ corresponding to the origin. Restricting $\Sigma_1$ if necessay we can assume that $\Sigma_1$ is given by the equation $z=0$.

\medskip

(1) From the above local expression (\ref{eq - metric0}) for the metric we see that the surface $\Sigma_0$, defined by the equation $y=0$, is totally geodesic and is orthogonal to $\Sigma_1.$  

\medskip

(2) Suppose there is a third totally geodesic surface $\Sigma_2$ containing $p$ which is not tangent to $\xi$ at $p$. 

We first treat the case when $\Sigma_2$ is not orthogonal to $\xi$ at $p.$
We will show that the function $\tau$ is constant in a neighborhood of $p.$  This will conclude the proof as this means that, in the coordinates introduced above,  $\theta(x)= \alpha x +\beta$ for some constants $\alpha$ and $\beta$. 
If $\alpha=0,$ the metric $g$ is flat. If $\alpha\ne 0,$ then it is straightforward to check that $g$ has constant sectional curvature 
$\alpha^2$.

We still denote by $\Sigma_2$ the component of $W\cap \Sigma_2$ containing $p$. Restricting $W$ and replacing $\Sigma_2$ by an open subset of it  if necessary, we can assume the intersection $\Sigma_1\cap\Sigma_2$ is connected. 
For $i=1,2,$ denote by $N_i$  a unit normal to $\Sigma_i$. As before we introduce the vector field $T_i$ tangent to $\Sigma_i$ and the real valued function $\nu_i$ on $\Sigma_i$ by the orthogonal decomposition 
$$\xi= T_i + \nu_i N_i.$$
We again use the same notations to denote the extensions of $N_i, T_i$ and $\nu_i$ to $W$ using the flow of $\xi.$ Note that along $\Sigma_1\cap\Sigma_2$ the vectors $N_1$ and $N_2$ are independent, so, up to restricting $W$ if necessary, we can assume that their extensions are also pointwise independent. In the same way, as 
$\Sigma_1$ and $\Sigma_2$ are distinct, $T_1$ and $T_2$ are independent along $\Sigma_1\cap\Sigma_2$ and we can assume their extensions are independent in $W.$

\medskip

Suppose that  $\xi$, $T_1$ and $T_2$ are linearly independent in an open set $U\subset W.$  It follows from Proposition \ref{prop4} and Remark \ref{coordinate} that $\tau$ does not depend on $T_1$, that is,  $T_1(\tau)=0$. In the same way $T_2(\tau)=0$. As $\xi(\tau)=0,$ see Lemma \ref{lem2}, we conclude that $\mathrm{grad}\,\tau=0$ in $U,$ where $\mathrm{grad}$ denotes the gradient on $M.$

\medskip

Suppose now that $\xi$, $T_1$ and $T_2$ are (pointwise) linearly dependent in some connected open set $V\subset W.$  Let $S$ denote the surface 
tangent to the distribution spanned by $\xi$ and $T_1$ which  passes through $p.$ From the expression of the metric \eqref{eq - metric0} obtained in Theorem \ref{theo2} using the surface $\Sigma_1,$ we see that the coordinate $x$ is the signed distance function to the surface $S.$ As we are assuming $\xi$, $T_1$ and $T_2$ are dependent in $V,$ we conclude that we obtain the same coordinate function $x$ when we use $\Sigma_2$ in Theorem \ref{theo2}.  The Killing fields $T_1$ and $T_2$ are tangent to the integral surfaces of the distribution spanned by $\xi$ and $T_1$, that is, the level surfaces of the coordinate function $x,$ and so are Killing fields on each of them. Moreover their norms depend only on  $x$ (see Remark \ref{coordinate}). For each $x,$ the level surface corresponding to $x$ is flat, that is, locally euclidean. The Killing fields  $\xi, T_1$ and $T_2$ on such a surface correspond to constant   fields under an  isometry with  an open subset of the Euclidean plane since they have constant norms. Therefore we have in $V$  a relation of the form
$$\xi= \alpha_1(x) T_1 + \alpha_2(x) T_2.$$
 We next show that the functions $\alpha_1$ and $\alpha_2$ are actually constant. 

As $\xi$, $T_1$ and $T_2$ are Killing fields, we have for any vector field $Y$:
$$ \alpha_1^\prime(x)\langle T_1,Y\rangle +  \alpha_2^\prime(x)\langle T_2, Y\rangle + Y(\alpha_1)\langle T_1,\partial_x\rangle + Y(\alpha_2)\langle T_2, \partial_x\rangle=0.$$
Taking successively $Y= T_1$ and $Y=T_2$ we get
$$ \langle \alpha_1'(x) T_1 + \alpha_2'(x) T_2, T_1 \rangle = \langle \alpha_1'(x) T_1 + \alpha_2'(x) T_2, T_2 \rangle = 0.$$

Since $T_1$ and $T_2$ are independent,  we conclude that $ \alpha_1^\prime(x)= \alpha_2^\prime(x)=0,$ that is,  $\alpha_1$ and $\alpha_2$ are constants. 

Replacing $T_i$ by $\xi -\nu_i N_i,$ for $i=1,2,$ 
in the decomposition $\xi= \alpha_1 T_1 + \alpha_2 T_2 $ we get
\begin{equation}\label{xi}
\xi= \gamma_1 \nu_1 N_1 +\gamma_2\nu_2N_2,
\end{equation}
where $ \gamma_i={\alpha_i}/({\alpha_1+\alpha_2-1}), i=1,2.$ We have:
$$\widetilde\nabla_{\partial_x}\xi = \gamma_1\nu_1^\prime(x)N_1+\gamma_2\nu_2^\prime(x)N_2.$$
As $\nu_i=\cos\theta_i(x)$, we have $\nu_i^\prime(x)=-\theta_i^\prime(x) \sin\theta_i(x)=\tau \sin\theta_i(x),\, i=1,2.$
So
$$\widetilde\nabla_{\partial_x}\xi =\tau\{\gamma_1 \sin\theta_1(x) N_1 +\gamma_2 \sin\theta_2(x) N_2\}.$$
Taking the inner product of both sides with $\xi$ we obtain:
$$\tau\{\gamma_1\sin\theta_1(x)\nu_1  +\gamma_2\sin\theta_2(x) \nu_2\}=0.$$
That is,
$$\tau\{\gamma_1\sin2\theta_1(x)+\gamma_2\sin2\theta_2(x)\}=0.$$
Suppose $\tau$ does not vanish in some open set $V_0\subset V.$ Then on $V_0:$
\begin{equation}\label{angles}
\gamma_1\sin2\theta_1(x)+\gamma_2\sin2\theta_2(x)=0.
\end{equation}
Taking the derivative we get:
 \begin{equation}\label{angles2}
\gamma_1\cos2 \theta_1(x)+\gamma_2\cos 2\theta_2(x)=0.
\end{equation}
 By (\ref{xi}), $(\gamma_1,\gamma_2)\ne (0,0),$  so the determinant of the system in the unknowns $\gamma_1$ and $\gamma_2$ formed by equations (\ref{angles}) and (\ref{angles2}) has to vanish, that is, 
$$ \sin 2(\theta_2(x)-\theta_1(x))=0.$$
However, the quantity $\theta_2(x)-\theta_1(x),$ which is a constant since $ \theta_1^\prime(x)=-\tau=\theta_2^\prime(x),$ is neither equal to $0$ nor to $\pm\pi/2.$ Indeed, otherwise this would imply that $\Sigma_2$ locally coincides with $\Sigma_1$ or $\Sigma_0$, but this contradicts the assumption that $\Sigma_2$ is a totally geodesic surface through $p$, different from $\Sigma_0$ and $\Sigma_1$.  Consequently $\tau \equiv 0$ in $V.$

Summarizing  we have shown that $\mathrm{grad}\,\tau =0$ on an open dense set in a neighborhood of $p$ in  $M^3.$ Therefore $\tau$ is constant near $p.$ This concludes the proof of (2) when $\Sigma_2$ is not orthogonal to $\xi$ at $p.$

\medskip

Suppose now that $\Sigma_2$ is orthogonal to $\xi$ at $p.$ Let $S_0\subset \Sigma_2$ denote the subset where $\Sigma_2$ is not orthogonal to $\xi.$ By the above argument, the function $\tau$ is locally constant on $S_0.$ Suppose now that 
$\Sigma_2$ is orthogonal to $\xi$ in an open set $S_1\subset \Sigma_2.$ So $\xi$ is the unit normal to $\Sigma_2$ on $S_1.$ As $\Sigma_2$ is totally geodesic, by the formula in Lemma \ref{lem2}, we have $\tau\equiv 0$ on $S_1.$ 
Consequently, denoting by $\mathrm{grad}^{\Sigma_2}$ the gradient on $\Sigma_2,$ we have $\mathrm{grad}^{\Sigma_2} \tau=0$ on an open dense subset of $\Sigma_2$ and so $\tau$ is constant on $\Sigma_2.$ As $\xi$ is transversal to $\Sigma_2$ and $\xi(\tau)=0,$ we conclude again that $\tau$ is constant in a neighborhood of $p$ in $M^3.$ This  concludes the proof of (2). 

\medskip

(3) Suppose $\tau(p)=0,$ that is, $\theta'(0)=0.$ Then from Proposition \ref{prop4}  the surface given by $x=0$ is totally geodesic. Conversely, suppose there  is a connected totally geodesic surface $\Sigma$ through $p$ which is tangent to $\xi$ at $p.$ We may assume $\Sigma$ is contained in the coordinate neighborhood where the metric on $M$ takes the form (\ref{eq - metric0}). 
By the same arguments as in (2), we get that $\tau$ is  constant on any connected  open  subset of $\Sigma$ where $\xi$ is not tangent to $\Sigma.$ Moreover  $\tau$ vanishes on any open subset of $\Sigma$ where $\xi$
is tangent to $\Sigma$ (see the proof of (2) in Theorem \ref{theo2}). As previously, we conclude that  $\tau$ is constant on $\Sigma.$

Denote by $\pi$ the projection on the $x-$axis. We consider three cases:
 
 - {\it First case: }  $I:=\pi(\Sigma)$ contains  an open interval containing $0.$ It follows that $\tau$ (which depends only on $x$) is constant in a neighborhood of $p.$ So $M$ is, near $p,$ flat  or has constant positive curvature. The second possibility is excluded by hypothesis and consequently $\tau$ is identically zero near $p.$ 

- {\it Second case:} $I=\{0\},$ that is, $\Sigma\subset \{x=0\}.$ From the equations in Proposition \ref{prop4}, we see that the surface $\{x=0\}$ is totally geodesic if and only if $\theta'(0)=0,$ that is, if and only if $\tau(p)=0.$ 

- {\it Third case:} $ \pi(\Sigma)\subset [0,+\infty)$ or $\pi(\Sigma)\subset (-\infty,0].$ This means the surface $\Sigma$ is on one side of  the surface $\{x=0\}.$  The extrinsic curvature of the surface $\{x=0\}$ is $K_{ext}= -(\theta'(0))^2= -\tau(p)^2$
as is seen from Proposition \ref{prop4}. It is therefore a saddle surface if $\tau(p)\neq 0$ and we are led in this case to a contradiction since it is a general fact that a 
totally geodesic surface tangent to a saddle surface at a point cannot lie on one side of it near the tangency point. So necessarily $\tau(p)=0.$ 
\end{proof}

%%%%%%%%%%%%%%%%%%%%%%%%%%%%%%%%%%%%%%%%%%%%%%%%%%%%%%%%%%%%%%%%%%%%%%%%%%%%%%%%%%%%%%%%%%%%%%%%%%%%

\section{Properties of the three-spaces}

In this section we shall discuss some properties of
three-dimensional spaces $M^3$ with a unit Killing field $\xi$ that
admit totally geodesic surfaces which are neither orthogonal nor tangent to $\xi.$ From Theorem \ref{theo2} we know that
such a manifold locally admits a metric of type
\begin{equation} \label{eq - metric}
g = dx^2 + \sin^2\theta(x) dy^2 + \cos^2\theta(x) dz^2,
\end{equation}
with $\xi = \partial_y + \partial_z$.

The following result, which can be checked through straightforward
computations, states that these three-spaces admit Riemannian submersions
onto a surface with a Killing field. 

\begin{proposition} \label{prop5}
Given $M^3$ as above, consider a surface $M^2$ with local
coordinates $(u,v)$ and metric
$$ du^2 + \frac 14 \sin^2(2\theta(u)) dv^2. $$
Then the mapping $\pi : M^3 \to M^2 : (x,y,z) \mapsto (u,v)=(x,y-z)$
is a Riemannian submersion whose fibers are integral curves of the
unit Killing field $\xi = \partial_y + \partial_z$. Remark that the
Gaussian curvature of $M^2$ is $K = 4(\theta')^2 -
2\cot(2\theta)\theta''$.
\end{proposition}

Let us now study some global properties of $M^3$. In particular we
want to investigate which manifolds admit a smooth metric which, in
local coordinates, is given by \eqref{eq - metric}.

We will first recall two lemmas from \cite{P} on a class of more
general doubly warped products
\begin{equation} \label{eq - doubly warped}
(I \times \S^p \times \S^q, \ dx^2 + \varphi^2(x) g_{\S^p} +
\psi^2(x) g_{\S^q}),
\end{equation}
where $I \subseteq \R$ is an open interval and $g_{\S^p}$ and
$g_{\S^q}$ are the standard Riemannian metrics on $\S^p$ and $\S^q$.

\begin{lemma}[\cite{P}] \label{lem3}
If $\varphi:(0,b)\to(0,\infty)$ is smooth and $\varphi(0)=0$, then
the metric in \eqref{eq - doubly warped} is smooth at $x=0$ if and
only if $\varphi^{(even)}(0)=0$, $\varphi'(0)=1$, $\psi(0)>0$ and
$\psi^{(odd)}(0)=0$. In this case, the topology near $x=0$ is
$\R^{p+1} \times \S^q$.
\end{lemma}
\begin{lemma}[\cite{P}] \label{lem4}
If $\varphi:(0,b)\to(0,\infty)$ is smooth and $\varphi(b)=0$, then
the metric in \eqref{eq - doubly warped} is smooth at $x=b$ if and
only if $\varphi^{(even)}(b)=0$, $\varphi'(b)=-1$, $\psi(b)>0$ and
$\psi^{(odd)}(b)=0$. In this case, the topology near $x=b$ is also
$\R^{p+1} \times \S^q$.
\end{lemma}

These results allow us to prove that a smooth metric of type
\eqref{eq - metric} exists on the simply connected manifolds $\S^3$,
$\S^2 \times \R$ and $\R^3$.

\begin{proposition} \label{prop6}
If $\theta : [0,b]\to\R$ is a smooth function such that
$\theta^{-1}\{0\}=\{0\}$ and $\theta^{-1}\{\pi/2\}=\{b\}$, then the
metric \eqref{eq - metric} defines a smooth metric on $\S^3$ if and
only if $\theta'(0) = \theta'(b) = 1$ and $\theta^{(2k)}(0) =
\theta^{(2k)}(b) = 0$ for any positive integer $k$.
\end{proposition}

\begin{proof}
It follows from the assumptions on $\theta$ that
$\theta(0,b)=(0,\pi/2)$. Hence, the functions $\varphi = \sin\theta$
and $\psi = \cos\theta$ are strictly positive on $(0,b)$. Lemma
\ref{lem3} and Lemma \ref{lem4} yield that \eqref{eq - metric} then
gives rise to a smooth metric on $\S^3$ if and only if the
conditions of Lemma \ref{lem3} are satisfied at $x=0$ and the
conditions of Lemma \ref{lem4} are satisfied at $x=b$, with the
roles of $\varphi$ and $\psi$ interchanged.

It is easy to see that the condition $\varphi'(0)=1$ is equivalent
to $\theta'(0)=1$ and that $\psi(0)>0$ is automatically satisfied.
Similarly $\psi'(b)=-1$ if and only if $\theta'(b)=1$ and
$\varphi(b)>0$ is automatically satisfied. The remaining conditions
are thus
$$ \varphi^{(even)}(0)=0, \ \psi^{(odd)}(0)=0, \ \psi^{(even)}(b)=0, \ \varphi^{(odd)}(b)=0. $$
After a computation and using that $\theta'(0)=\theta'(b)=1$, one
sees that these conditions are equivalent to
$\theta^{(2k)}(0)=\theta^{(2k)}(b)=0$ for any integer $k>0$.
\end{proof}

\begin{remark}
The function $\theta(x)=x$ satisfies the conditions given in
Proposition \ref{prop6}. In this case, the metric \eqref{eq -
metric} corresponds to the standard metric on $\S^3$ and the
Riemannian submersion of Proposition \ref{prop5} is the classical
Hopf fibration.
\end{remark}

\begin{proposition} \label{prop7}
If $\theta : [0,b]\to [0,\infty)$ is a smooth function such that
$\theta^{-1}\{0\}=\{0,b\}$ and $\theta^{-1}\{\pi/2\}=\varnothing$,
then the metric \eqref{eq - metric} defines a smooth metric on
$\S^2\times\R$ if and only if $\theta'(0) = -\theta'(b) = 1$ and
$\theta^{(2k)}(0) = \theta^{(2k)}(b) = 0$ for any non-negative
integer $k$.
\end{proposition}

\begin{proof}
Remark that the functions $\varphi=\sin\theta$ and $\psi=\cos\theta$
are positive on $(0,b)$. Hence, \eqref{eq - metric} defines a smooth
metric on $\S^2\times\R$ if and only if the conditions of Lemma
\ref{lem3} and Lemma \ref{lem4} are satisfied. We can now proceed in
an analogous way as in the proof of Proposition \ref{prop6} to
obtain the result.
\end{proof}

\begin{proposition} \label{prop8}
If $\theta : \R \to \R$ is a smooth function such that
$\theta^{-1}\{ k\pi \, | \, k \in\Z \} = \theta^{-1}\{ \pi/2 + k\pi
\, | \, k \in\Z \} = \varnothing$, then \eqref{eq - metric} defines
a smooth metric on $\R^3$, which is moreover complete.
\end{proposition}

\begin{proof}
It is clear that the metric is smooth under the given assumptions.
To prove completeness, we may assume that $\theta(x) \in (0,\pi/2)$
for all $x\in\R$. Now let $\gamma : [0,T) \to \R^3 : t \mapsto
(\gamma_1(t),\gamma_2(t),\gamma_3(t))$ be a curve which diverges to
infinity, i.e., such that $\gamma_1(t)^2 + \gamma_2(t)^2 +
\gamma_3(t)^2$ tends to infinity if $t$ tends to $T$. We have to
prove that the length of this curve with respect to the metric
\eqref{eq - metric},
$$ L(\gamma) = \int_0^T \sqrt{(\gamma_1'(t))^2 + \sin^2(\theta(\gamma_1(t))) (\gamma_2'(t))^2 + \cos^2(\theta(\gamma_1(t))) (\gamma_3'(t))^2} \, dt, $$
is infinite. Therefore, we consider two cases.

First, assume that $\gamma_1$ is unbounded. In this case, we have
$$ L(\gamma) \geq \int_0^T |\gamma_1'(t)| \, dt \, \geq \, \lim_{t \to T} | \gamma_1(t) - \gamma_1(0) | \, = \, \infty. $$

Next, assume that $\gamma_1$ is bounded. In that case the function
$\theta(\gamma_1(t))$ is bounded away from $0$ and $\pi/2$ and hence
there exists a real constant $c > 0$ such that
$\sin(\theta(\gamma_1(s))) \geq c$ and $\cos(\theta(\gamma_1(s)))
\geq c$. This implies that
$$ L(\gamma) \geq c \int_0^T \sqrt{(\gamma_2'(t))^2 + (\gamma_3'(t))^2} \, dt \, = \, \infty. $$
The last equality is due to the fact that the integral
appearing on the left hand side is the Euclidean length of the
projection of the curve $\gamma$ onto the $(y,z)$-plane. Since
$\gamma$ diverges to infinity but $\gamma_1$ is bounded, this
projection must have infinite length.
\end{proof}

It is possible to check, using for instance our Theorem \ref{sphere}, that in the examples of Proposition 6 and Proposition 7, through  the points where $x=0$ or $x=b,$ there is no totally geodesic surface 
which is not tangent to the unit Killing field unless the function $\theta'$ is constant in a neighborhood 
of $x=0$ and $x=b,$ respectively. This is not a mere coincidence. We actually have the following global 
result.

 \begin{theorem}\label{global}
Let $M$ be a connected and simply connected complete Riemannian three-manifold carrying a unit Killing  field $\xi.$  Suppose that: 

\begin{itemize}
\item [(1)] no open subset of $M$ has constant non-negative curvature,
\item[(2)] through each point of $M$ there passes a totally geodesic surface which is neither orthogonal nor tangent to $\xi.$
\end{itemize}

Then $M$ is isometric to $\R^3$ endowed with the metric: 
$$ds^2 = dx^2 + \sin^2\theta(x) dy^2 + \cos^2\theta(x) dz^2,$$
where $\theta : \R \to (0, \pi/2)$ is a smooth function
whose derivative $\theta'$ is not constant on any interval. Moreover  $\xi =\partial_y+\partial_z.$
\end{theorem}

\begin{proof}

Let $p_0$ be a fixed point in $M.$ By Theorems \ref{theo2} and \ref{sphere},  $\xi$ admits an orthogonal  decomposition, $\xi =X_1+X_2,$   in a neighborhood of $p_0,$  where $X_1$ and $X_2$ are two Killing fields which commute and have no zeros. 
Moreover this decomposition is unique up to ordering of $X_1$ and $X_2.$ By the same theorems, given any point $p\in M$ and any continuous path $\gamma$ joining $p_0$ to $p,$ we  can extend continuously this decomposition along $\gamma$ till the point $p.$  As $M$ is simply connected, by a standard monodromy argument, the decomposition we obtain at $p$ is independent on the choice of the path $\gamma.$ We get in this way a global orthogonal decomposition $\xi= X_1+X_2,$ where $X_1$ and $X_2$ are now global smooth Killing vector fields on $M$ which commute and have no zeros. 

Denote by $u$ a unit vector field on $M$ orthogonal to $X_1$ and $X_2.$ Such a global smooth vector field exists since $M$ is orientable. From the proof of (3)-(ii) of Theorem \ref{theo2}, we know that the distribution spanned by $u$ and $X_1$ is integrable and its integral surfaces are totally geodesic. Therefore the manifold $M$ admits a foliation $\mathcal F$ by totally geodesic surfaces. Let $\mathcal{F}^{\perp}$ be the orthogonal foliation, that is, the foliation by the orbits of $X_2.$ By a result of Carri\`ere and Ghys \cite{CG}, $(\mathcal{F},\mathcal{F}^{\perp})$ is a product. This means   there is a diffeomorphism between $M$ and the product $ \Sigma\times\R,$ where $\Sigma$ is any fixed leaf of $\mathcal F,$ sending the leaves of $\mathcal{F}$ to $\Sigma\times\{\ast\}$ and those of $\mathcal{F}^{\perp}$ to $\{\ast\}\times \R.$  Denote by $z$ the coordinate on the $\R$ factor. Under this diffeomorphism, the vector field $X_2$ therefore corresponds to 
the field $f(z)\partial_z$ for some function $f.$ So, up to reparameterizing the $\R$ factor, we can assume
that  $X_2$ correponds to the field $\partial_z.$  It is clear that  $\Sigma$ is simply connected and is therefore, topologically, either a plane or a sphere.  $X_1$ is vector field on $\Sigma$ which has no zeros, so $\Sigma$ is topologically a plane. It is moreover not difficult to check that $\Sigma$ is complete.

Fix an orientation on $\Sigma$ and denote by $J$ the rotation over 
90 degrees in $T\Sigma.$ As in the proof of (3) of Theorem \ref{theo2}, we consider on $\Sigma$ the fields $X_1$ and $JX_1$
which commute, $[X_1, JX_1]=0,$ and are complete since they have bounded norms and $\Sigma$ is complete. It follows that we can find a global chart for $\Sigma$ with domain $\R^2$  and $\partial_x= JX_1$ and $\partial_y=X_1$ for the standard coordinates $(x,y)$ on $\R^2.$  We include a proof of this fact for completeness.  Let ${(\phi_x)}_{x\in\R}$ and ${(\psi_y)}_{y\in\R}$ be the flows of $JX_1$ and $X_1,$ respectively. 
Consider the mapping: 
$$ (x,y)\in \R^2\to F(x,y)=(\phi_x\circ\psi_y) (p_0)\in \Sigma.$$
$F$ is a local diffeomorphism with $(dF)(\partial_x)= JX_1$ and $(dF)(\partial_y)=X_1.$  We will show it is a global diffeomeorphism, this will  provide the global chart we want.  

- {\it $F$ is one-to-one:} by contradiction, suppose $(x_1,y_1)$ and $(x_2,y_2)$ are distinct points in $\R^2$ with $F(x_1,y_1)=F(x_2,y_2).$ Assume that $x_1=x_2$ and $y_1\ne y_2,$ then the orbit of $X_1$ through $\phi_{x_1}(p_0)$ will be closed
and  will bound  a disk in $\Sigma$ inside which necessarily $X_1$ will have a zero, which is a contradiction. The case $x_1\ne x_2$ and $y_1=y_2$ is similar. Assume now that $x_1\ne x_2$ and $y_1\ne y_2.$ Set $p_1=(\phi_{x_1}\circ\psi_{y_1})(p_0)=(\phi_{x_2}\circ\psi_{y_2})(p_0).$ The orbit of $JX_1$ through $p_0$ and the orbit of $X_1$ through  $p_1$ intersect at two distinct points, namely $\phi_{x_1}(p_0)$ and $\phi_{x_2}(p_0).$ However an orbit of $JX_1$ can intersect an orbit of $X_1$ at most once. Indeed let $\gamma_1$ and $\gamma_2$ be orbits of $X_1$ and $JX_1,$ respectively. Assume they intersect more than once. Then there will be a bounded disk $\Omega$ in $\Sigma$ with boundary the union of an arc $\alpha_1\subset \gamma_1$ and an arc $\alpha_2\subset \gamma_2$ with common endpoints $p$ and $q.$ We assume that along $\alpha_2,$ the field $X_1$ points into $\Omega.$ The case when $X_1$ points outside $\Omega$ can be treated in a similar way. Consider any point $q_1$ on $\alpha_2$ distinct from $p$ and $q.$ The half orbit $\beta:=\{\psi_t(q_1), t>0\}$
of $X_1$ 
starting from $q_1$ will  be entirely contained in $\Omega.$ It follows from the  Poincar\'e-Bendixon theorem, see for instance \cite{HS}, that  the accumulation set of $\beta$ must contain a zero or a closed orbit of $X_1$, which is again a contradiction.

-{\it F is onto:} since $F$ is a local diffeomorphism, the image $F(\R^2)$ is open in $\Sigma,$ so to conclude it is enough to see that it  is closed. Let $(x_n,y_n), n\in \Bbb N,$ be a sequence of points in $\R^2$ with $F(x_n,y_n)\to p_{\infty}\in \Sigma$ as $n\to \infty.$ For $\epsilon>0$ small enough, the mapping $(x,y)\in(-\epsilon, \epsilon)
\times (-\epsilon, \epsilon) \to (\phi_x\circ\psi_y) (p_{\infty})\in \Sigma$ is an embedding with image an open neighborhood $V$ of $p_{\infty}.$ For $n$ fixed and big enough, we have $F(x_n,y_n)\in V$ and so there is $(x,y) \in (-\epsilon, \epsilon)
\times (-\epsilon, \epsilon)$ such that $(\phi_x\circ\psi_y) (p_{\infty})= (\phi_{x_n}\circ\psi_{y_n}) (p_0).$ Therefore 
$p_{\infty}= \phi_{x_n-x}\circ\psi_{y_n-y} (p_0)$ lies in $ F(\R^2).$

In the global coordinates $(x,y,z),$ like in the proof of (3) in Theorem \ref{theo2} and with the same notations, the metric on $M$ writes:
$$ds^2=(1-\nu(x)^2)(dx^2+dy^2)+ \nu(x)^2 dz^2,$$
where $\nu(x)=\|X_2\|^2$ is a function of $x$ alone. We now make the change of coordinate 
$\bar x(x)= \int \sqrt{1-\nu(x)^2} dx.$ By the completeness of the metric $g,$ the function $\bar x$ is a bijection from $\R$ onto $\R.$  
Setting $\nu(x)=\cos \theta (\bar x)$ for some smooth function $\theta: \R\to (0,\pi/2),$ the metric writes  in the global coordinates $(\bar x, y, z):$
$$ds^2 = d\bar{x}^2 + \sin^2\theta(\bar x) dy^2 + \cos^2\theta(\bar x) dz^2.$$
The condition (1) in the statement means precisely that $\theta'$ is not constant on  any interval (see the proof of (2) in Theorem \ref{sphere}).
\end{proof}

%%%%%%%%%%%%%%%%%%%%%%%%%%%%%%%%%%%%%%%%%%%%%%%%%%%%%%%%%%%%%%%%%%%%%%%%%%%%%%%%%%%%%%%%%%%%%%%%%%%%%%%%%%%%%%%%%%%%%%%

\end{document}